\documentclass[11pt]{amsart}

\usepackage{amsmath,amssymb,amsthm,mathtools}
\usepackage[margin=1in]{geometry}
\usepackage[colorlinks=true,citecolor=blue,linkcolor=blue,urlcolor=blue,
pdftitle={Graphical discreteness, Coxeter doublings and generalized polygons},
pdfauthor={Xing-Yu Hu}]{hyperref}
\newtheorem{theorem}{Theorem}[section]
\newtheorem{lemma}[theorem]{Lemma}
\newtheorem{proposition}[theorem]{Proposition}
\newtheorem{corollary}[theorem]{Corollary}
\newtheorem{question}[theorem]{Question}
\newtheorem*{theoremA}{Theorem A}
\newtheorem*{theoremB}{Theorem B}
\newtheorem*{corollaryC}{Corollary C}
\newtheorem*{theoremD}{Theorem D}
\newtheorem*{theoremE}{Theorem E}
\newtheorem*{corollaryF}{Corollary F}
\theoremstyle{definition}
\newtheorem{definition}[theorem]{Definition}
\theoremstyle{remark}
\newtheorem{remark}[theorem]{Remark}

\newcommand{\Aut}{\operatorname{Aut}}
\newcommand{\Ch}{\operatorname{Ch}}
\newcommand{\proj}{\operatorname{proj}}
\newcommand{\cC}{\mathcal C}
\newcommand{\KG}{\operatorname{KG}}
\newcommand{\Isom}{\operatorname{Isom}}
\newcommand{\QI}{\operatorname{QI}}
\newcommand{\Col}{\operatorname{Col}}

\title[Graphical discreteness and Coxeter doublings]
{Graphical discreteness, Coxeter doublings and generalized polygons}
\author{Xing-Yu Hu}
\address{School of Mathematics and Statistics, Hanjiang Normal University,
No.~18 Beijing South Road, Shiyan 442000, Hubei, China}
\email{huxingyu@hjnu.edu.cn}

\begin{document}

\begin{abstract}
Graphical discreteness is not preserved under quasi-isometry in general
\cite[Corollary~1.3]{MSSW}. Nevertheless, for every finite thick generalized
$m$-gon $\Gamma$ with $m\geq3$, graphical discreteness is constant among
finitely generated groups quasi-isometric to the right-angled Coxeter group
$W_\Gamma$. Such a group $\Lambda$ is graphically discrete if and only if
$\Gamma$ is nonflexible, meaning that no nontrivial graph automorphism of
$\Gamma$ fixes a closed star pointwise \cite[p.~1]{White}. The proof
combines Fuchsian-building quasi-isometric rigidity with a metric-strata
argument that recovers the standard Coxeter Cayley graph from the building
metric. For finite-rank Coxeter systems, the standard Cayley-graph
automorphism group is compact-by-discrete if and only if it is discrete,
and both conditions are equivalent to nonflexibility of the defining
diagram. Further results use one-vertex doublings to obstruct graphical
discreteness, characterize when the full automorphism groups of standard
chamber graphs of graph products of finite groups are compact-by-discrete,
and give applications to finite projective planes and right-angled Coxeter
groups with Menger curve boundary. The arguments do not determine whether any nonflexible finite
thick generalized polygon exists.
\end{abstract}

\subjclass[2020]{Primary 20F65; Secondary 20F55, 51E12, 51E15, 51E24, 22D05, 05C25}
\keywords{graphical discreteness, Coxeter group, Coxeter doubling,
generalized polygon, graph product, Fuchsian building, right-angled building}

\maketitle

\section{Introduction}

Graphical discreteness is not preserved by quasi-isometry in general
\cite[Corollary~1.3]{MSSW}. The main result concerns a natural family of
quasi-isometry classes on each of which it is nevertheless constant. If
$\Gamma$ is a finite thick generalized $m$-gon with $m\geq3$ and $\Lambda$ is
any finitely generated group quasi-isometric to the right-angled Coxeter group
$W_\Gamma$, then
\[
 \Lambda\text{ is graphically discrete}
 \quad\Longleftrightarrow\quad
 \Gamma\text{ is nonflexible}.
\]
Thus graphical discreteness is constant among finitely generated groups
quasi-isometric to $W_\Gamma$, even though it is not a quasi-isometry invariant
in general. This is Theorem~D below.

Graphical discreteness was introduced by Margolis, Shepherd, Stark and
Woodhouse as a rigidity condition governing all geometric graph actions of a
finitely generated group \cite[\S3.1]{MSSW}. A finitely
generated group $\Lambda$ is graphically discrete if $\Aut(X)$ is
compact-by-discrete for every connected locally finite graph $X$ on which
$\Lambda$ acts properly and cocompactly \cite[Proposition~3.6(3)]{MSSW}.
The word ``compact'' is essential. A
nondiscrete automorphism group can still be compact-by-discrete, so
nondiscreteness of $\Aut(X)$ for a single geometric graph action is not by
itself an obstruction.

For the special case $\Lambda=W_\Gamma$, the nonflexible implication can
already be recovered from the Davis-complex rigidity results of Haglund--Paulin
\cite[Th\'eor\`eme~5.12]{HaglundPaulin}, the right-angled finite-index
formulation in \cite[Remark~4.12]{Shepherd}, Xie's Fuchsian-building rigidity
theorem \cite{Xie}, recalled in \cite[Theorem~2.4]{BoundsXie}, the Bounds--Xie
realization \cite[Theorem~4.5]{BoundsXie}, and the finite-index quasi-isometry
criterion \cite[Theorem~3.27]{MSSW}. Theorem~1.11(ii) of
\cite{Shepherd} also gives a short route to the nonflexible implication for
arbitrary $\Lambda$. Every such $\Lambda$ is abstractly commensurable with
$W_\Gamma$. If $H\leq W_\Gamma$ has finite index, inclusion is a quasi-isometry
and induces $\QI(H)\cong\QI(W_\Gamma)$. Under this isomorphism, the
left-multiplication image of $H$ has finite index in the left-multiplication
image of $W_\Gamma$. The left-multiplication image of $W_\Gamma$ has finite
index in $\QI(W_\Gamma)$, as verified explicitly in the proof of
Theorem~\ref{thm:polygon-classification}. Moreover, $H$ is non-elementary hyperbolic
and hence tame, so \cite[Theorem~3.27]{MSSW} implies that $H$ is graphically discrete.
Let $\Lambda_0\leq\Lambda$ and $H\leq W_\Gamma$ be isomorphic finite-index
subgroups supplied by commensurability. Then $\Lambda_0$ is graphically
discrete, and so is $\Lambda$, since graphical discreteness passes from a
finite-index subgroup to the ambient group directly from the definition.
This route cannot recover the flexible implication, since graphical
discreteness need not pass to finite-index subgroups
\cite[Corollary~1.3]{MSSW}. For that direction,
\cite[Corollary~1.2]{BoundsXie} and the metric-strata lemma supply a common
Coxeter Cayley-graph witness throughout the quasi-isometry class.

The flexible implication in Theorem~D is proved uniformly for every finitely
generated group in the quasi-isometry class. Theorem~A upgrades the standard
Coxeter Cayley-graph criterion from discreteness to compact-by-discreteness.
Theorem~B turns one-vertex doublings into additional finite-index witnesses.
The later generalized-polygon and graph-product results make these criteria
explicit in finite-geometric and building-theoretic families.

The remaining results give obstructions and structural criteria for Coxeter groups
and graph products of finite groups, with applications to finite geometry. To state
Theorem~A, recall that the defining diagram of a Coxeter system $(W,S)$ has
vertex set $S$, where distinct $s,t\in S$ are joined by a
labelled edge exactly when $m_{st}<\infty$. It is \emph{flexible} if there are
$s\in S$ and a nontrivial label-preserving diagram automorphism $\phi$ that
fixes $s$ and every vertex joined to $s$ \cite[p.~1]{White}.

\begin{theoremA}
Let $(W,S)$ be a Coxeter system with $S$ finite. For the standard undirected
Cayley graph $\cC(W,S)$, the following conditions are equivalent.
\begin{enumerate}
\item $\Aut(\cC(W,S))$ is compact-by-discrete.
\item $\Aut(\cC(W,S))$ is discrete.
\item the defining diagram of $(W,S)$ is nonflexible.
\end{enumerate}
Consequently, a Coxeter group admitting a flexible Coxeter system is not
graphically discrete.
\end{theoremA}

Haglund--Paulin's restriction isomorphism and rigid-case vertex-stabilizer
calculation give the nonflexible direction for finite-rank Coxeter systems
\cite[Corollaire~5.8 and Th\'eor\`eme~5.12]{HaglundPaulin}. Their flexible-side
construction assumed word hyperbolicity. White later introduced a first-$s$
Cayley-graph automorphism for flexible diagrams
\cite[Definition~21 and Proposition~26]{White}. Berlai and Ferov independently
established the finite-rank nondiscreteness criterion and, more generally,
characterized uncountable vertex stabilizers in the standard Coxeter Cayley
graph \cite[Corollary~B and Theorem~A]{BerlaiFerov}.
To rule out compact-by-discreteness, the first-$s$ construction is verified
directly, and conjugates of the resulting automorphism produce arbitrarily
large finite vertex-stabilizer orbits. The stabilizer-orbit criterion of
Margolis--Shepherd--Stark--Woodhouse
\cite[Theorem~3.10(1),(2)]{MSSW} then rules out compact-by-discreteness, not
merely discreteness.

Theorem~A uses the standard Cayley graph of the Coxeter system. A different
geometric witness arises from an index-two Coxeter subgroup. For a finite simplicial graph $\Gamma$, write $V\Gamma$ and $E\Gamma$ for its vertex and edge sets. For $v\in V\Gamma$, the Doubling Lemma
\cite[Lemma~2.7]{DHW} uses the graph $D_v\Gamma$ formed by gluing two copies
of $\Gamma\setminus\{v\}$ along their copies of
$\operatorname{lk}_\Gamma(v)$, and realizes $W_{D_v\Gamma}$ as an index-two
subgroup of $W_\Gamma$.

\begin{theoremB}
If $D_v\Gamma$ is flexible for some $v\in V\Gamma$, then the
right-angled Coxeter group $W_\Gamma$ is not graphically discrete.
\end{theoremB}

Conjugation by $v$ acts on the standard generating set of
$W_{D_v\Gamma}$ by the deck involution of $D_v\Gamma$. Hence $W_\Gamma$ acts
geometrically on the
standard Cayley graph of the doubled subgroup. If the doubled diagram is
flexible, Theorem~A shows that the full automorphism group of this graph is
not compact-by-discrete. This argument also recovers the obstruction from
proper partial conjugations.

\begin{corollaryC}
If $\Gamma\setminus\operatorname{st}_\Gamma(v)$ has at least two connected
components for some $v\in V\Gamma$, then $W_\Gamma$ is not graphically discrete. In
particular, the existence of a proper partial conjugation obstructs graphical
discreteness.
\end{corollaryC}

The first assertion of Corollary~C is proved below as
Corollary~\ref{cor:separating-star}. The partial-conjugation formulation follows
immediately afterward. Corollary~\ref{cor:partial-double}, which applies when a
link vertex misses a component of the complement of the relevant closed star,
is called the \emph{link-gap criterion}. Corollary~\ref{cor:dominated-star},
which applies when distinct vertices $s,v$ satisfy
$\operatorname{st}(s)\subseteq\operatorname{st}(v)$ and
$\operatorname{st}(v)\neq V\Gamma$, is called the \emph{dominated-star
criterion}. The doubling obstruction is strictly stronger than these two
criteria. Corollary~\ref{cor:asymmetric-family}
constructs an infinite family $\Theta_m$, $m\geq2$, of pairwise
nonisomorphic asymmetric graphs without dominated vertices for which the
corresponding groups are one-ended, hyperbolic and not graphically discrete,
although the link-gap criterion fails at every vertex.

The generalized-polygon theorem announced at the outset uses the
Fuchsian-building metric of Bounds and Xie. They equip the Davis complex of
$W_\Gamma$ with a regular piecewise hyperbolic metric that makes it a Fuchsian
building \cite[Theorem~4.5]{BoundsXie}. The classification obtained here is
as follows.

\begin{theoremD}
Let $\Gamma$ be a finite thick generalized $m$-gon, where $m\geq3$, and let
$\Lambda$ be any finitely generated group quasi-isometric to $W_\Gamma$. Then
\[
 \Lambda\text{ is graphically discrete}
 \quad\Longleftrightarrow\quad
 \Gamma\text{ is nonflexible}.
\]
\end{theoremD}

Thus graphical discreteness, although not a quasi-isometry invariant in general
\cite[Corollary~1.3]{MSSW}, is constant on each of these quasi-isometry classes.

By the Feit--Higman theorem, the finite thick case can occur only for
$m\in\{3,4,6,8\}$ \cite{FeitHigman}. The uniform formulation $m\geq3$ is
retained in Theorem~D.

Theorem~D does not assert the existence of a nonflexible finite thick
generalized $m$-gon with $m\geq3$. The existence question is posed in
Section~\ref{sec:questions}.
Proposition~\ref{prop:structural-flexibility} uses standard root-group and
elation actions to exclude two broad families. Every finite thick Moufang
generalized polygon is flexible. The same holds for every finite thick elation
generalized quadrangle and its dual. For a finite projective plane with
nonflexible incidence graph, the full automorphism group of each one-vertex
double is
described explicitly, and
nonflexibility passes to every double. Proposition~\ref{prop:plane-perspectivity} gives the precise
projective-plane criterion. The incidence graph is nonflexible if and only if
the plane has no nontrivial perspectivity. Baer's involution theorem, as stated in \cite[p.~878]{Cofman}, gives
a further restriction. At nonsquare order, the collineation group of such a plane has odd order and is
therefore solvable by the Feit--Thompson theorem
\cite[Chapter~I, \S1, p.~775]{FeitThompson}. Thus a finite perspectivity-free projective
plane would give a graphically discrete hyperbolic right-angled Coxeter group
whose defining graph and all one-vertex doubles are nonflexible.

For a finite simplicial graph $\Gamma$ with vertex set $S$ and a family
$(G_s)_{s\in S}$ of nontrivial finite groups, let $G_\Gamma$ denote the
corresponding graph product \cite[Chapter~3]{Green}. Its standard chamber graph
$\cC_\Gamma$ has vertex set $G_\Gamma$, with distinct vertices $g,h$ adjacent when
$g^{-1}h\in G_s\setminus\{1\}$ for some $s\in S$
\cite[Definitions~2.6 and~3.1]{Shepherd}. Put
\[
 T=\{s\in S:|G_s|>2\},\qquad \Gamma_0=\Gamma[S\setminus T].
\]
Here $\Gamma[A]$ denotes the subgraph induced by $A$. By the definition
of flexibility, a graph with empty vertex set is nonflexible. A vertex is
\emph{universal} if it
is adjacent to every other vertex. Kubena and Thomas determined the discreteness dichotomy for
automorphism groups of regular right-angled buildings
\cite[Theorem~1]{KT}. Theorem~E gives the compact-by-discrete classification
for the standard chamber graph.

\begin{theoremE}
The following conditions are equivalent.
\begin{enumerate}
\item $\Aut(\cC_\Gamma)$ is compact-by-discrete.
\item $\Aut(\cC_\Gamma)$ is discrete.
\item every vertex of $T$ is universal in $\Gamma$ and $\Gamma_0$ is
nonflexible.
\end{enumerate}
In particular, if $G_\Gamma$ is graphically discrete, then every vertex of
$T$ is universal in $\Gamma$ and $\Gamma_0$ is nonflexible.
\end{theoremE}

Finally, Theorem~A applies to a family of hyperbolic right-angled Coxeter
groups with Menger curve boundary. Let
$O_n=\KG(2n-1,n-1)$ be the $n$-th odd graph, where $\KG(N,r)$ denotes the
Kneser graph whose vertices are the $r$-element subsets of an $N$-element set
and whose edges join disjoint subsets \cite[p.~5]{EkinciGauci}.

\begin{corollaryF}
For every $n\geq3$, the right-angled Coxeter group $W_{O_n}$ is hyperbolic,
has Gromov boundary homeomorphic to the Menger curve, and is not graphically
discrete.
\end{corollaryF}

Theorems~A, B, D and E are proved below as
Theorem~\ref{thm:flexible}, Theorem~\ref{thm:doubling-obstruction},
Theorem~\ref{thm:polygon-classification} and
Theorem~\ref{thm:chamber-classification}, respectively. Corollary~F is proved
below as Corollary~\ref{cor:menger-odd}. Corollary~C is located above.

\section{Preliminaries}

An isometric action is called \emph{geometric} if it is proper, cocompact and
continuous in the sense of \cite[Definition~2.22]{MSSW}. Finitely generated
groups carry the discrete topology \cite[Remark~2.23]{MSSW}. Unless stated otherwise,
geometric graph actions are on connected locally finite graphs. Defining graphs for
graph products and right-angled Coxeter groups are finite simplicial graphs and may be
disconnected. Automorphism groups of locally finite graphs carry the topology of
pointwise convergence on vertices. For a proper metric space $Y$, the group
$\Isom(Y)$ is equipped here with the compact-open topology. With this topology
it acts properly and continuously on $Y$
\cite[\S2.4, immediately after Remark~2.23]{MSSW}.

\begin{definition}[{\cite[\S3.1 and Proposition~3.6(3)]{MSSW}}]
A locally compact group $H$ is \emph{compact-by-discrete} if it has a compact
normal subgroup $K$ such that $H/K$ is discrete. Using the equivalent graph
action formulation, a finitely generated group $\Lambda$ is \emph{graphically discrete} if $\Aut(X)$ is
compact-by-discrete for every geometric action $\Lambda\curvearrowright X$ on
a connected locally finite graph.
\end{definition}

Margolis--Shepherd--Stark--Woodhouse give the following compact-by-discrete
criterion.

\begin{theorem}[{\cite[Theorem~3.10(1),(2)]{MSSW}}]
\label{thm:MSSW-orbits}
Let $X$ be a connected locally finite graph, and let $H\leq\Aut(X)$ be a
closed subgroup acting cocompactly on $X$.  Then $H$ is compact-by-discrete if and only if the
cardinalities
\[
 |H_x\cdot y|,\qquad x,y\in V(X),
\]
are uniformly bounded.
\end{theorem}

Only the contrapositive will be used. In this setting, unbounded stabilizer
orbits rule out compact-by-discreteness.

A second criterion comes from quasi-isometries. A finitely generated group is
\emph{tame} in the sense of \cite[Definition~2.28]{MSSW} if quasi-isometries
with fixed constants are uniformly close whenever they lie at finite distance
from one another. Every acylindrically hyperbolic group is tame by
\cite[Example~2.29]{MSSW}. In particular, every non-elementary hyperbolic
group is tame. The same authors also prove the following graphical-discreteness
criterion.

\begin{theorem}[{\cite[Theorem~3.27]{MSSW}}]
\label{thm:MSSW-QI}
Let $\Lambda$ be a finitely generated tame group. If the image of the
left-multiplication homomorphism
\[
 \Lambda\longrightarrow\QI(\Lambda)
\]
has finite index, then $\Lambda$ is graphically discrete.
\end{theorem}

\begin{lemma}\label{lem:closed-subgroup}
Every closed subgroup of a compact-by-discrete locally compact group is
compact-by-discrete.
\end{lemma}

\begin{proof}
Let $K\trianglelefteq H$ be compact with $H/K$ discrete, and let $L\leq H$ be
closed.  The subgroup $K$ is open in $H$, so $L\cap K$ is a compact open
normal subgroup of $L$.  Thus $L/(L\cap K)$ is discrete.
\end{proof}

The next permanence fact concerns finite direct factors.

\begin{lemma}\label{lem:finite-factor}
If $F$ is finite and $\Lambda$ is finitely generated, then $F\times\Lambda$ is
graphically discrete if and only if $\Lambda$ is graphically discrete.
\end{lemma}

\begin{proof}
Suppose first that $\Lambda$ is graphically discrete. The subgroup $\Lambda$
has finite index in $F\times\Lambda$. Restricting any geometric action of $F\times\Lambda$ to
$\Lambda$ is still geometric, so the full automorphism group of the underlying
graph is compact-by-discrete.

Conversely, if $\Lambda$ has a geometric action on $X$ witnessing failure of
graphical discreteness, let $F\times\Lambda$ act on the same graph through the
projection to $\Lambda$. Its stabilizers are $F\times\Lambda_x$, hence
finite. The action is proper and cocompact. The same graph $X$ is therefore a witness for
$F\times\Lambda$.
\end{proof}

\section{Main results}

\subsection{Flexible Coxeter systems}

Let $(W,S)$ be a Coxeter system of finite rank.  The standard undirected
Cayley graph $\cC(W,S)$ has vertex set $W$ and an edge from $w$ to $ws$ for
each $w\in W$ and $s\in S$.

\begin{definition}[{\cite[p.~1]{White}}]
The defining diagram of $(W,S)$ is
\emph{flexible} if there are $s\in S$ and
a nontrivial label-preserving automorphism $\phi$ of the diagram such that
$\phi(s)=s$ and $\phi(r)=r$ whenever $m_{sr}<\infty$. For
$(W_\Gamma,V\Gamma)$, this says exactly that some nonidentity graph
automorphism of $\Gamma$ fixes a closed star pointwise. Thus nonflexibility in
the right-angled case is precisely the star-rigid condition used in
\cite[Remark~4.12]{Shepherd}.
\end{definition}

\begin{theorem}\label{thm:flexible}
Let $(W,S)$ be a Coxeter system with $S$ finite. The following conditions are
equivalent.
\begin{enumerate}
\item $\Aut(\cC(W,S))$ is compact-by-discrete.
\item $\Aut(\cC(W,S))$ is discrete.
\item the defining diagram is nonflexible.
\end{enumerate}
If the diagram is flexible, then $W$ is not graphically discrete.
\end{theorem}

\begin{proof}
Assume that the diagram is nonflexible. Haglund--Paulin's restriction
isomorphism and rigid-case theorem show that the stabilizer of $1$ in
$\Aut(\cC(W,S))$ is isomorphic to the finite Coxeter-diagram automorphism
group $G(W,S)$
\cite[Corollaire~5.8 and Th\'eor\`eme~5.12]{HaglundPaulin}. Point stabilizers
are open in the pointwise-convergence topology, so this finite stabilizer
forces $\Aut(\cC(W,S))$ to be discrete. Thus (3) implies~(2).
Condition~(2) implies~(1), since a discrete group is compact-by-discrete. It
remains to show that flexibility rules out~(1).

Choose $s$ and $\phi$ as in the definition. Since $\phi$ is nontrivial, there
is $t\in S$ such that $u=\phi(t)\neq t$. If $m_{st}<\infty$, then $\phi(t)=t$,
a contradiction. Thus $m_{st}=\infty$. Since $\phi(s)=s$ and $\phi$ is
injective, $u\neq s$. Label preservation then gives $m_{su}=\infty$. In
particular, the flexible case forces $W$ to be infinite.

First consider a Cayley-graph automorphism. Following White's first-$s$
construction \cite[Definition~21]{White}, define $\Psi$ on reduced words as
follows. If a reduced word contains $s$, write it as
\[
 w=w_1sw_2,
\]
where $w_1$ contains no $s$, and set
\[
 \Psi(w)=\phi(w_1)sw_2.
\]
If $w$ contains no $s$, set $\Psi(w)=\phi(w)$.

First note that the displayed image word is reduced. Suppose otherwise. By
Tits' word-reduction theorem \cite[Theorem~3.4.2]{Davis}, choose a shortening
sequence and stop at its first deletion. All preceding steps are braid moves.
Pull them back across the first-$s$ cut. A move before the cut pulls back under
$\phi^{-1}$, one after the cut is unchanged, and a move meeting the cut has
other generator $r$ with $m_{sr}<\infty$, hence $\phi(r)=r$. This gives a
reduced word $w'$ braid-equivalent to $w$ such that $\Psi(w')$ contains an
adjacent equal pair. Before the cut, $\phi^{-1}$ pulls this pair back to $w'$.
After the cut it is unchanged. If it meets the cut, injectivity and
$\phi(s)=s$ force an adjacent $ss$ in $w'$. Each case contradicts reducedness.

It remains to show that the definition is independent of the chosen reduced
expression. By Matsumoto's theorem, again
\cite[Theorem~3.4.2]{Davis}, two reduced expressions for the same element
differ by braid moves. Consider one such move. If its alternating block lies before
the first $s$, applying $\phi$ to the block carries the braid relation to the
corresponding braid relation because $\phi$ preserves Coxeter labels. If the
block lies after the first $s$, the same braid move is left unchanged. Finally,
if the block contains the first $s$, then its other letter $r$ satisfies
$m_{sr}<\infty$, so $\phi(r)=r$. The entire block is therefore fixed by
$\phi$, even when an odd braid move changes the position of the first $s$.
Thus $\Psi$ respects every braid move and descends to a map on $W$.
Replacing $\phi$ by $\phi^{-1}$ gives the inverse of $\Psi$.

The map also preserves adjacency. If $wr$ is reduced, then either $w$ already
contains $s$, in which case $\Psi(wr)=\Psi(w)r$, or $w$ contains no $s$. In
the latter case, $\Psi(wr)=\Psi(w)s$ when $r=s$, and
$\Psi(wr)=\Psi(w)\phi(r)$ when $r\neq s$. If $wr$ is not reduced, apply the
same argument to $w=(wr)r$. Hence $\Psi$ and its inverse preserve edges, so
$\Psi\in\Aut(\cC(W,S))$ and $\Psi(1)=1$.

For $n\geq0$, put
\[
 w_n=(st)^n s,
 \qquad
 \Phi_n=L_{w_n}\circ\Psi\circ L_{w_n}^{-1}.
\]
Since $m_{st}=\infty$, the displayed word for $w_n$ is reduced and
$w_n^{-1}=w_n$. Its first letter is $s$, so $\Psi(w_n)=w_n$. Therefore each
$\Phi_n$ fixes the identity.

For $k\geq1$, put $y_k=(st)^k$. If $0\leq n<k$, then
\[
 y_k=w_n\,t(st)^{k-n-1},
 \qquad
 w_n^{-1}y_k=t(st)^{k-n-1}.
\]
The first-$s$ prescription gives
\[
 \Psi\bigl(t(st)^{k-n-1}\bigr)=u(st)^{k-n-1}.
\]
When $n=k-1$, the argument of $\Psi$ is just $t$, so the formula applies as
written. Consequently
\[
 \Phi_n(y_k)=(st)^n(su)(st)^{k-n-1}.
\]
The standard parabolic subgroup on $\{s,t,u\}$ splits as
\[
 W_{\{s,t,u\}}=\langle s\rangle*W_{\{t,u\}},
\]
because no finite Coxeter relation joins $s$ to $t$ or $u$. In free-product
normal form, $y_k$ has $k$ successive $W_{\{t,u\}}$-blocks all equal to $t$,
whereas $\Phi_n(y_k)$ has the $(n+1)$-st such block equal to $u$. Since
$u\neq t$, the $k+1$ points
\[
 y_k,\ \Phi_0(y_k),\ldots,\Phi_{k-1}(y_k)
\]
are pairwise distinct. Hence
\[
 \left|\Aut(\cC(W,S))_1\cdot y_k\right|\geq k+1
 \qquad(k\geq1).
\]

The full automorphism group acts cocompactly on $\cC(W,S)$ because it
contains the vertex-transitive left action of $W$. Theorem~\ref{thm:MSSW-orbits}
therefore implies that $\Aut(\cC(W,S))$ is not compact-by-discrete. Finally,
the left action of $W$ on this graph is geometric, so the same graph witnesses
that $W$ is not graphically discrete.
\end{proof}

\subsection{Generalized polygons and quasi-isometric rigidity}

By \cite[\S1.2]{Shepherd}, a \emph{generalized $m$-gon} is a connected
bipartite graph of diameter $m$ and girth $2m$, and it is \emph{thick} if every
vertex has degree at least three. The equivalent incidence-geometric formulation,
in which ordinary $m$-gons are the apartments, is recalled in
\cite[\S2.2]{DeMedtsVanMaldeghem}. In particular, a thick generalized $3$-gon
is a projective plane \cite[\S1.2]{Shepherd}. By the Feit--Higman theorem, a finite
thick generalized $m$-gon with $m\geq3$ can occur only for
$m\in\{3,4,6,8\}$ \cite{FeitHigman}.
Let $\Gamma$ be a finite thick generalized $m$-gon with $m\geq3$. Since
$\Gamma$ is bipartite and hence triangle-free, its Davis complex
$\Sigma_\Gamma$ is a two-dimensional
CAT(0) cube complex with vertex set $W_\Gamma$ and $1$-skeleton
$\cC(W_\Gamma,V\Gamma)$
\cite[Proposition~7.3.4, Example~7.3.6 and Theorem~12.2.1(i)]{Davis}.

Bounds and Xie replace every Euclidean square of $\Sigma_\Gamma$ with a
regular hyperbolic quadrilateral of angle $\pi/m$. The resulting proper
metric space, denoted here by $X_\Gamma$ while Bounds and Xie retain
$\Sigma_\Gamma$, is a Fuchsian building \cite[Theorem~4.5]{BoundsXie}. Since
\cite[Definition~2.3]{BoundsXie} requires each edge to lie in at least three
chambers, $X_\Gamma$ is thick in the building-theoretic sense. The
$W_\Gamma$-action remains proper and cocompact.

\begin{theorem}\label{thm:polygon-classification}
Let $\Gamma$ be a finite thick generalized $m$-gon, where $m\geq3$, and let
$\Lambda$ be a finitely generated group quasi-isometric to $W_\Gamma$.  Then
\[
 \Lambda\text{ is graphically discrete}
 \quad\Longleftrightarrow\quad
 \Gamma\text{ is nonflexible}.
\]
\end{theorem}

\begin{lemma}
\label{lem:metric-cubical}
Every metric isometry of $X_\Gamma$ preserves the cubical vertices, the open
cubical edges and the open squares.  Restriction to the vertex set therefore
defines a continuous injective homomorphism
\[
 \Isom(X_\Gamma)\longrightarrow
 \Aut\bigl(\cC(W_\Gamma,V\Gamma)\bigr).
\]
\end{lemma}

\begin{proof}
The three metric strata are distinguished by their spaces of directions. At
an interior point of a square, the space of directions is a circle. If a
cubical edge has type $s\in V\Gamma$, then it is contained in one square for
each $t\in\operatorname{lk}_\Gamma(s)$. Hence an interior point of that edge
has space of directions the suspension of a
$\deg_\Gamma(s)$-point set. Thickness gives $\deg_\Gamma(s)\geq3$, so this is
a graph with exactly two branch points. At a cubical vertex, the space of
directions is the metric realization of $\Gamma$, with every edge of length
$\pi/m$. The graph $\Gamma$ has more than two vertices and every one of them
has degree at least three, so this link has more than two branch points. Thus
the three strata are metric invariants.

It follows that every isometry preserves the cubical vertices and open edges.
The connected components of the complement of those strata are the open
squares. A cubical square is determined by its four vertices, so an isometry
fixing every cubical vertex preserves every square setwise. Its restriction to
a regular hyperbolic square fixes all four vertices and is therefore the
identity on that square. Restriction to the
vertex set is thus injective. The cubical vertex set is locally finite and
therefore discrete.
For every finite set of vertices $A$, there is $\varepsilon>0$ such that an
isometry moving each point of $A$ by less than $\varepsilon$ fixes $A$
pointwise.  Hence the preimage of every basic pointwise stabilizer in the
target is open in the compact-open topology, which proves continuity.
\end{proof}

\begin{proof}[Proof of Theorem~\ref{thm:polygon-classification}]
By Bounds--Xie quasi-isometric rigidity, $\Lambda$ acts geometrically on
$X_\Gamma$ \cite[Corollary~1.2]{BoundsXie}. Lemma~\ref{lem:metric-cubical}
shows that this action preserves the cubical vertex and edge strata and hence
restricts to an action on the standard Cayley graph
$\cC(W_\Gamma,V\Gamma)$. This restricted action is proper. Finite vertex
sets are compact subsets of $X_\Gamma$, so their transporter sets are finite
for the original proper action of the discrete group $\Lambda$. It is also cocompact.
Indeed, if $K\subseteq X_\Gamma$ is compact and $\Lambda K=X_\Gamma$, then $K$
meets the locally finite cubical vertex set in only finitely many points. Since
$\Lambda$ preserves that vertex set, those points meet every $\Lambda$-orbit
of cubical vertices.

If $\Gamma$ is flexible, Theorem~\ref{thm:flexible} shows that
$\Aut(\cC(W_\Gamma,V\Gamma))$ is not compact-by-discrete. The preceding
geometric action therefore witnesses that $\Lambda$ is not graphically
discrete.

Assume that $\Gamma$ is nonflexible. By Theorem~\ref{thm:flexible},
$\Aut(\cC(W_\Gamma,V\Gamma))$ is discrete. Lemma~\ref{lem:metric-cubical}
then implies that
$G=\Isom(X_\Gamma)$ is discrete. Let $J\leq G$ be the image of the action
homomorphism $\Lambda\to G$. Properness gives finite point stabilizers, hence a
finite kernel. Moreover, $J$ has finite index in $G$.
Indeed, choose a compact set $K\subseteq X_\Gamma$ with $JK=X_\Gamma$, and
fix $x\in X_\Gamma$. Since $X_\Gamma$ is proper, the full isometry group acts
properly, so
\[
 F=\{g\in G:gx\in K\}
\]
is compact and hence finite because $G$ is discrete. For each $g\in G$, choose
$j\in J$ with $j^{-1}gx\in K$. Then $j^{-1}g\in F$, so $G=JF$.

Every self-quasi-isometry of $X_\Gamma$ is at finite distance from an
isometry \cite[Theorem~1.1]{BoundsXie}. Hence the natural map
$G\to\QI(X_\Gamma)$ is surjective. Since $J$ has finite index in $G$, its
image has finite index in $\QI(X_\Gamma)$. An orbit map
$o:\Lambda\to X_\Gamma$ is a quasi-isometry \cite[p.~58]{MSSW}, and
$o\circ L_\lambda=\lambda\circ o$ for left multiplication $L_\lambda$.
Under the isomorphism $\QI(\Lambda)\cong\QI(X_\Gamma)$ induced by $o$,
the image of the left-multiplication homomorphism
$\Lambda\to\QI(\Lambda)$ is the image of $J$. It therefore has finite
index.

Finally, $\Gamma$ has no induced four-cycle, so $W_\Gamma$ is hyperbolic
\cite[Corollary~12.6.3]{Davis}. Each vertex has three pairwise nonadjacent
neighbours, which generate a special subgroup $C_2*C_2*C_2$. Hence
$W_\Gamma$ is non-elementary. Consequently every finitely generated group
quasi-isometric to $W_\Gamma$, including $\Lambda$, is non-elementary
hyperbolic and hence tame \cite[Example~2.29]{MSSW}. Theorem~\ref{thm:MSSW-QI}
completes the proof.
\end{proof}

\begin{remark}
Corollary~4.2 of \cite{MSSW} states that uniform lattices in thick locally
finite hyperbolic buildings are not graphically discrete. There is therefore a
conditional tension with Theorem~\ref{thm:polygon-classification}. For a finite
thick generalized $m$-gon $\Gamma$, the standard $W_\Gamma$-action on
$X_\Gamma$ is faithful and geometric. Since $X_\Gamma$ is proper,
\cite[Lemma~2.24]{MSSW} realizes $W_\Gamma$ as a uniform lattice in
$\Isom(X_\Gamma)$. Hence, if Corollary~4.2 holds as stated, applying
Theorem~\ref{thm:polygon-classification} with $\Lambda=W_\Gamma$ forces every
finite thick generalized $m$-gon with $m\geq3$ to be flexible. Equivalently,
the existence of a nonflexible such $\Gamma$ would give a counterexample to
Corollary~4.2. In particular, by
Proposition~\ref{prop:plane-perspectivity}, Corollary~4.2 would imply that every
finite projective plane of order at least two admits a nontrivial
perspectivity. Thus Corollary~4.2 and an affirmative answer to
Question~\ref{q:polygon-existence} cannot both hold.

The displayed proof of Corollary~4.2 uses the assertion that any two apartments
containing a fixed chamber can be interchanged by a global building
automorphism fixing that chamber pointwise. The standard building axiom
recalled in \cite[\S2]{Xie} supplies an isomorphism between two apartments
fixing their intersection. By itself, it does not supply an extension of that
apartment isomorphism to a global building automorphism. Thus the displayed
proof requires an additional apartment-extension assumption. This observation
concerns that proof only. No claim about the truth value of Corollary~4.2 is
made here.

The proof of Theorem~\ref{thm:polygon-classification} is independent of
Corollary~4.2 and of any apartment-extension property. It uses
\cite[Corollary~1.2]{BoundsXie}, Lemma~\ref{lem:metric-cubical},
\cite[Theorem~1.1]{BoundsXie}, and \cite[Theorem~3.27]{MSSW}.
\end{remark}

\begin{proposition}
\label{prop:structural-flexibility}
The incidence graph is flexible for every finite thick Moufang generalized
polygon and for every finite thick generalized quadrangle that is an elation
generalized quadrangle or the dual of one.
\end{proposition}

\begin{proof}
For the Moufang case, let $\alpha=(x_0,\ldots,x_m)$ be a root. The root
group $U_\alpha$ consists of the elations fixing every element incident with
one of the interior vertices $x_1,\ldots,x_{m-1}$
\cite[\S3]{DeMedtsVanMaldeghem}. For a thick Moufang polygon, $U_\alpha$ is
nontrivial and acts regularly on the apartments through $\alpha$
\cite[\S3]{DeMedtsVanMaldeghem}. Hence any nonidentity element of
$U_\alpha$ fixes the closed star of each interior vertex $x_i$ pointwise.

For the elation case, let $p$ be the base point and let $E$ be an elation
group about $p$. By definition, $E$ fixes every line through $p$ and acts
regularly on the points not collinear with $p$
\cite[Chapter~8, pp.~138--139]{PayneThas}. If the quadrangle has order
$(s,t)$, the set of points not collinear with $p$ has cardinality $s^2t$.
Thickness gives $s,t\geq2$, so this cardinality is at least eight. Hence $E$ is
nontrivial. A nonidentity element of
$E$ fixes the closed star of the point vertex $p$ pointwise. Interchanging
points and lines gives the dual statement.
\end{proof}

Consequently, any example realizing the graphically discrete case of
Theorem~\ref{thm:polygon-classification} must be non-Moufang. In the
quadrangle case it can be neither an elation generalized quadrangle nor the
dual of one. These restrictions do not amount to a classification of all
finite generalized polygons.

\subsection{Coxeter doublings and finite-index witnesses}

Throughout this subsection, $\Gamma$ is a finite simplicial graph. For a
vertex $v\in V\Gamma$, write
\[
 \operatorname{st}_\Gamma(v)=\{v\}\cup\operatorname{lk}_\Gamma(v)
\]
for its closed star.

Fix $v\in V\Gamma$. The \emph{double} $D_v\Gamma$ consists of two copies of
$\Gamma\setminus\{v\}$ glued along $\operatorname{lk}_\Gamma(v)$
\cite[\S2.2, immediately before Lemma~2.7]{DHW}. Equivalently,
vertices in the link of $v$ have one copy, while every vertex outside
$\operatorname{st}_\Gamma(v)$ has two copies.  Define
\[
 \varepsilon_v:W_\Gamma\longrightarrow C_2
\]
by sending $v$ to the nontrivial element and every other standard generator
to the identity, and put $H_v:=\ker(\varepsilon_v)$. The Doubling Lemma
\cite[Lemma~2.7]{DHW} states that $W_{D_v\Gamma}$ occurs as an index-two
subgroup of $W_\Gamma$. For the explicit realization used below, apply
Reidemeister--Schreier with transversal $\{1,v\}$. For each $r\neq v$ the
resulting generators are $r$ and $r^*=vrv$. They coincide when
$r\in\operatorname{lk}_\Gamma(v)$. If $r\notin\operatorname{st}_\Gamma(v)$,
the special subgroup $W_{\{v,r\}}\cong C_2*C_2$ shows that $r^*\neq r$.
Rewriting the involution and commutation relators gives
\[
 r^2=(r^*)^2=1,\qquad [r,s]=[r^*,s^*]=1
 \quad\text{for }\{r,s\}\in E\Gamma,\ r,s\neq v.
\]
When a vertex lies in $\operatorname{lk}_\Gamma(v)$, the equality $r^*=r$
makes the link common to the two copies. No cross-copy commutation relation
occurs between two vertices outside the closed star. Thus this is
exactly the Coxeter presentation of $W_{D_v\Gamma}$. Hence
$H_v\cong W_{D_v\Gamma}$, with a link vertex represented by $r$ and the two
copies of each $r\notin\operatorname{st}_\Gamma(v)$ represented by $r$ and
$vrv$.

\begin{theorem}\label{thm:doubling-obstruction}
Let $\Gamma$ be a finite simplicial graph and let $v\in V\Gamma$. If
$D_v\Gamma$ is flexible, then $W_\Gamma$ is not graphically discrete.
\end{theorem}

\begin{proof}
Let
\[
 X=\cC\bigl(W_{D_v\Gamma},V(D_v\Gamma)\bigr)
\]
be the standard Cayley graph of $H_v$.  Conjugation by $v$ fixes each standard
generator from $\operatorname{lk}_\Gamma(v)$ and interchanges the two
generators $r$ and $vrv$ associated to each vertex
$r\notin\operatorname{st}_\Gamma(v)$.  It therefore induces the deck
involution of $D_v\Gamma$ and an automorphism $\alpha$ of $X$.

The splitting $W_\Gamma=H_v\rtimes\langle v\rangle$ now gives an action on
$X$ by
\[
 (h v^\epsilon)\cdot x=h\alpha^\epsilon(x)
 \qquad(h,x\in H_v,\ \epsilon\in\{0,1\}).
\]
The subgroup $H_v$ acts simply transitively on the vertices, and the
stabilizer of the identity in $W_\Gamma$ is $\langle v\rangle$.
Thus the action of $W_\Gamma$ on $X$ is proper and cocompact.

If $D_v\Gamma$ is flexible, Theorem~\ref{thm:flexible} implies that
$\Aut(X)$ is not compact-by-discrete.  Hence this geometric action on $X$
witnesses that $W_\Gamma$ is not graphically discrete.
\end{proof}

For a component $C$ of $\Gamma\setminus\operatorname{st}_\Gamma(v)$, let
$\tau_C$ interchange the two copies of $C$ in $D_v\Gamma$ and fix every
other vertex.  Since distinct components have no edges between them,
$\tau_C$ is a graph automorphism.

\begin{corollary}
\label{cor:separating-star}
If $\Gamma\setminus\operatorname{st}_\Gamma(v)$ has at least two connected
components for some $v\in V\Gamma$, then $W_\Gamma$ is not graphically
discrete.
\end{corollary}

\begin{proof}
Choose distinct components $C$ and $C'$ and a vertex $r$ in one copy of
$C'$.  Every vertex in the closed star of $r$ in $D_v\Gamma$ belongs either to
that same copy of $C'$ or to the common link of $v$. The nontrivial
automorphism $\tau_C$ fixes all of these vertices.  Thus $D_v\Gamma$ is flexible, and
Theorem~\ref{thm:doubling-obstruction} applies.
\end{proof}

The component swap also has a group-theoretic interpretation. The standard partial
conjugation supported on $C$ is the automorphism described in
\cite[p.~1747]{GenevoisMartin}, namely
\[
 \chi_{v,C}(r)=
 \begin{cases}
  vrv,&r\in C,\\
  r,&r\notin C
 \end{cases}
 \qquad(r\in V\Gamma).
\]
Its restriction to
$H_v$ induces $\tau_C$ on the doubled Coxeter generating set.  The product
of the partial conjugations over all components is conjugation by $v$. When
there are at least two components, each individual factor has proper
support. Corollary~\ref{cor:separating-star} therefore shows that a proper
partial conjugation obstructs graphical discreteness.

A similar component swap can be used in some cases even when the complement
of the star is connected.

\begin{corollary}\label{cor:partial-double}
Suppose that $C$ is a connected component of
$\Gamma\setminus\operatorname{st}_\Gamma(v)$ and that some
$s\in\operatorname{lk}_\Gamma(v)$ has no neighbour in $C$.  Then
$W_\Gamma$ is not graphically discrete.
\end{corollary}

\begin{proof}
In $D_v\Gamma$, interchange the two copies of $C$ and fix every other
vertex.  There are no edges from $C$ to another component of
$\Gamma\setminus\operatorname{st}_\Gamma(v)$, and its two copies have
identical adjacency to the common link of $v$.  The interchange is therefore
a nontrivial graph automorphism. It fixes the closed star of $s$ pointwise
because $s$ has no neighbour in $C$.  Thus $D_v\Gamma$ is flexible, and
Theorem~\ref{thm:doubling-obstruction} applies.
\end{proof}

\begin{corollary}
\label{cor:dominated-star}
If distinct vertices $s,v\in V\Gamma$ satisfy
\[
 \operatorname{st}_\Gamma(s)\subseteq\operatorname{st}_\Gamma(v)
 \quad\text{and}\quad
 \operatorname{st}_\Gamma(v)\neq V\Gamma,
\]
then $W_\Gamma$ is not graphically discrete.
\end{corollary}

\begin{proof}
The star inclusion implies that $s\in\operatorname{lk}_\Gamma(v)$ and
that $s$ has no neighbour outside $\operatorname{st}_\Gamma(v)$. Since
$\operatorname{st}_\Gamma(v)\neq V\Gamma$, the complement has a vertex.
Apply Corollary~\ref{cor:partial-double} to any component of
$\Gamma\setminus\operatorname{st}_\Gamma(v)$.
\end{proof}

Corollary~\ref{cor:separating-star} detects examples missed by the link-gap
criterion in Corollary~\ref{cor:partial-double}, even when the defining graph
is asymmetric with no dominated vertices and the associated right-angled
Coxeter group is one-ended and hyperbolic.  For $m\geq2$, let
$\Theta_m$ be the union of the paths
\[
 P_0=s-v-t,\qquad
 P=s-a_1-a_2-a_3-a_4-a_5-t,\qquad
 Q_m=s-b_1-\cdots-b_m-t,
\]
and the additional path
\[
 R=a_1-c_1-c_2-c_3-a_4,
\]
where all displayed internal vertices are distinct. The graph $\Theta_m$ is
connected by construction.

\begin{corollary}\label{cor:asymmetric-family}
For $m\geq2$, the graph $\Theta_m$ is asymmetric and has no distinct
vertices $x,y$ satisfying
$\operatorname{st}_{\Theta_m}(x)\subseteq
\operatorname{st}_{\Theta_m}(y)$.  The groups $W_{\Theta_m}$ are pairwise
nonisomorphic, one-ended, hyperbolic and not graphically discrete.
Moreover, there is no triple $(x,C,y)$ in which $C$ is a component of
$\Theta_m\setminus\operatorname{st}_{\Theta_m}(x)$ and
$y\in\operatorname{lk}_{\Theta_m}(x)$ has no neighbour in $C$.
\end{corollary}

\begin{proof}
The vertices of degree three are precisely $s,t,a_1,a_4$. Every other
vertex has degree two. Suppressing the degree-two vertices gives a
four-vertex multigraph. The two parallel $s$--$t$ paths have lengths $2$ and
$m+1$, whereas the two parallel $a_1$--$a_4$ paths have lengths $3$ and $4$.
These two unordered pairs of lengths are distinct, so every automorphism
preserves the unordered pairs $\{s,t\}$ and $\{a_1,a_4\}$.
The two remaining paths joining these pairs have lengths one and two. The unique
length-one path singles out $s$ and $a_1$. Hence every automorphism fixes all
four branch vertices. The parallel paths have distinct lengths, so each is preserved and
all of its internal vertices are fixed. Thus $\Theta_m$ is asymmetric.

The cycles obtained from the two parallel $a_1$--$a_4$ paths and the two
parallel $s$--$t$ paths have lengths $7$ and $m+3$, respectively.  Every
remaining cycle uses the two single paths and one path from each parallel
pair, and hence has length $8$, $9$, $m+7$ or $m+8$.  Thus $\Theta_m$ has
girth $\min\{7,m+3\}\geq5$.  In particular it is triangle-free and has no induced
four-cycle, so $W_{\Theta_m}$ is hyperbolic \cite[Corollary~12.6.3]{Davis}.

Every vertex has degree at least two.  If
$\operatorname{st}(x)\subseteq\operatorname{st}(y)$ for distinct vertices,
then $x$ and $y$ are adjacent.  A second neighbour $z$ of $x$ would also be
adjacent to $y$, producing a triangle.  This proves that there are no
dominated vertices.

There is no separating clique. Since the graph is triangle-free, every
clique is a vertex or an edge. First consider deleting a vertex. If the deleted vertex
is internal to one of the six maximal branch-to-branch paths, that path is
replaced by at most two tails attached to its branch endpoints. The
suppressed multigraph remains connected after removing any one of its six
paths. If instead a branch vertex is deleted, the other branch vertices that
survive remain connected in the suppressed multigraph, and every remnant of
a path incident with the deleted vertex is a tail attached to its opposite
endpoint. Thus no vertex separates $\Theta_m$.

Next consider deleting the endpoints of an edge lying on a maximal path $J$. If neither
endpoint is a branch vertex, the surviving pieces of $J$ are tails attached
to the endpoints of $J$, while the suppressed multigraph with $J$ removed is
connected. If exactly one endpoint is a branch vertex, delete that branch
vertex in the suppressed multigraph. The remaining branch core is connected,
and all surviving path pieces attach to it. The only edge whose two endpoints
are branch vertices is $s-a_1$. After deleting $s$ and $a_1$, the vertices
$t$ and $a_4$ remain joined by the path through $a_5$, and every surviving
segment attaches to this connected core. Hence deleting a clique never disconnects $\Theta_m$. In a right-angled
Coxeter group, every clique generates a finite special subgroup. Therefore
\cite[Corollary~16]{MihalikTschantz} implies that $W_{\Theta_m}$ has at most
one end. Since $s$ and $t$ are nonadjacent, the special subgroup
$W_{\{s,t\}}\cong C_2*C_2$ is infinite. Thus $W_{\Theta_m}$ is infinite and
hence one-ended.

For every $x\neq v$, the graph
$\Theta_m\setminus\operatorname{st}_{\Theta_m}(x)$ is connected. Suppose
first that $x$ is internal to a maximal path. If neither neighbour of $x$ is
a branch vertex, removing the closed star of $x$ leaves at most two tails
attached to the connected branch core obtained by deleting that maximal
path. If exactly one neighbour is a branch vertex, remove that branch vertex
from the suppressed multigraph. The remaining branch core is connected, and
all surviving path pieces attach to it. The only internal vertex other than
$v$ adjacent to two branch vertices is $a_5$. Deleting
$\operatorname{st}(a_5)=\{a_4,a_5,t\}$ leaves $s$ joined to $a_1$, with
all surviving path segments attached to that pair. If $x$ is a branch
vertex, the surviving core can be described directly. For $x=s$, the vertices
$t$ and $a_4$ remain joined through $a_5$. For $x=t$, the vertices $s$ and
$a_1$ remain adjacent. For $x=a_1$, the vertices $t$ and $a_4$ remain joined
through $a_5$, while for $x=a_4$, the vertices $s,t,a_1$ remain connected.
Every remaining path segment attaches to the indicated core.
Since $\Theta_m$ is triangle-free and has minimum degree two, every $y\in\operatorname{lk}(x)$ has a second neighbour outside
$\operatorname{st}(x)$, and hence a neighbour in that unique component. For
$x=v$, the complement of the closed star has exactly two components,
\[
 A=\{a_1,\ldots,a_5,c_1,c_2,c_3\},
 \qquad B=\{b_1,\ldots,b_m\},
\]
and each of the two link vertices meets both components. The vertex $s$ is
adjacent to $a_1,b_1$, and $t$ is adjacent to $a_5,b_m$.  This proves the final assertion and shows that
Corollary~\ref{cor:partial-double} does not detect the family.

On the other hand, the displayed sets $A$ and $B$ are distinct components of
$\Theta_m\setminus\operatorname{st}(v)$, so
Corollary~\ref{cor:separating-star} proves that $W_{\Theta_m}$ is not
graphically discrete.  Finally,
$W_{\Theta_m}^{\mathrm{ab}}\cong(C_2)^{m+11}$, and hence the groups are
pairwise nonisomorphic.
\end{proof}

\subsection{Projective planes}

For projective planes, nonflexibility imposes additional structure. The first
result describes the automorphism groups of one-vertex doubles.

\begin{proposition}
\label{prop:projective-plane-double}
Let $\Gamma$ be the incidence graph of a finite projective plane of order
$q\geq2$, let $v\in V\Gamma$, and let $\delta_v$ be the deck involution of
$D_v\Gamma$. If $\Gamma$ is nonflexible, then
\[
 \Aut(D_v\Gamma)=\Aut(\Gamma)_v\times\langle\delta_v\rangle,
\]
where $\Aut(\Gamma)_v$ denotes the stabilizer of $v$ and acts diagonally
on the two copies. In particular,
$D_v\Gamma$ is nonflexible.
\end{proposition}

\begin{proof}
Fix $v$ and put $L=\operatorname{lk}_\Gamma(v)$.  In $D_v\Gamma$, every
vertex of $L$ has degree $2q$, while every other vertex has degree $q+1$.
Since $q\geq2$, the set $L$ is invariant under
$\Aut(D_v\Gamma)$.

The graph $\Gamma\setminus\operatorname{st}_\Gamma(v)$ is connected. By
point-line duality, assume that $v$ is a point. For distinct points
$x,y\neq v$, if $x,y,v$ are not collinear, the path $x-(xy)-y$ avoids the
star of $v$.  If $x,y,v$ are collinear, choose a point $z$ outside their common
line. Then
\[
 x-(xz)-z-(zy)-y
\]
is such a path.  Every line not through $v$ is incident with a remaining
point.  Thus deleting $L$ from $D_v\Gamma$ leaves exactly two connected
components, the two copies of
$\Gamma\setminus\operatorname{st}_\Gamma(v)$.

After composing with $\delta_v$ if necessary, every automorphism may be
assumed to preserve the two components. Since $D_v\Gamma$ is connected and
bipartite, an automorphism either preserves or interchanges its two bipartition
classes. The nonempty invariant set $L$ lies in one class, so the classes are
preserved. After restoring $v$, the induced automorphisms give elements
$g_1,g_2\in\Aut(\Gamma)_v$ with the same restriction to $L$.  Hence
$g_2^{-1}g_1$ fixes $\operatorname{st}_\Gamma(v)$ pointwise, and
nonflexibility gives $g_1=g_2$.  Consequently
\[
 \Aut(D_v\Gamma)=\Aut(\Gamma)_v\times\langle\delta_v\rangle,
\]
where the first factor acts diagonally.  The deck involution commutes with
this diagonal action and does not belong to it, so the product is direct.

Suppose an automorphism fixes the closed star of a vertex $r$ in
$D_v\Gamma$ pointwise. Its decomposition cannot involve $\delta_v$. If $r\notin L$, the deck
involution interchanges the components. If $r\in L$, it interchanges the two
sets of neighbours of $r$ outside $L$.  The automorphism is therefore
diagonal, induced by some $g\in\Aut(\Gamma)_v$.  If $r\notin L$, then $g$
fixes the corresponding closed star in $\Gamma$.  If $r\in L$, it fixes all
neighbours of $r$ other than $v$ and also fixes $v$, so again it fixes a
closed star in $\Gamma$.  Nonflexibility forces $g=1$.
\end{proof}

\begin{proposition}
\label{prop:plane-perspectivity}
Let $\Pi$ be a finite projective plane with incidence graph $\Gamma$. Then
$\Gamma$ is flexible if and only if $\Pi$ admits a nontrivial perspectivity.
\end{proposition}

\begin{proof}
A nontrivial perspectivity with centre $p$ fixes $p$ and every line through
$p$, hence fixes the closed star of the point vertex $p$ in $\Gamma$. Thus
$\Gamma$ is flexible.

Conversely, let $1\neq g\in\Aut(\Gamma)$ fix
$\operatorname{st}_\Gamma(v)$ pointwise. Since $\Gamma$ is connected and
bipartite and $g$ fixes $v$, the automorphism $g$ preserves the point and line
classes and hence induces a collineation of $\Pi$. If $v$ is a point, then
$g$ fixes every line through $v$, so $v$ is a centre of $g$. If $v$ is a
line, then $v$ is an axis of $g$. The standard centre-axis theorem for projective-plane collineations states
that a collineation has a centre if and only if it has an axis \cite[Theorem~4.9, p.~94]{HughesPiper}. Hence $g$ is a nontrivial central collineation, or equivalently an axial
collineation, and therefore a perspectivity.
\end{proof}

\begin{corollary}
If a finite projective plane of order $q\geq2$ is perspectivity-free and has
incidence graph $\Gamma$, then $W_\Gamma$ is graphically discrete and every
one-vertex double $D_v\Gamma$ is nonflexible.
\end{corollary}

\begin{proof}
By Proposition~\ref{prop:plane-perspectivity}, $\Gamma$ is nonflexible.
Apply Theorem~\ref{thm:polygon-classification} and
Proposition~\ref{prop:projective-plane-double}.
\end{proof}

Let $\Pi$ be a finite projective plane and let $\Col(\Pi)$ denote its
collineation group.

\begin{proposition}
\label{prop:plane-involutions}
Let $\Pi$ have order $q$ and incidence graph $\Gamma$. If $\Gamma$ is
nonflexible (equivalently, $\Pi$ is perspectivity-free) and $q$ is not a
perfect square, then $\Col(\Pi)$ has odd order and is solvable.
\end{proposition}

\begin{proof}
By Proposition~\ref{prop:plane-perspectivity}, $\Pi$ has no nontrivial
perspectivity. By Baer's involution theorem, an involution of a finite projective plane is either a perspectivity
or fixes a subplane of order $\sqrt q$ pointwise \cite[p.~878]{Cofman}. The latter case can occur
only when $q$ is a perfect square. Thus, under the
stated hypotheses, $\Col(\Pi)$ contains no involution. Cauchy's theorem then
implies that $|\Col(\Pi)|$ is odd. The Feit--Thompson theorem shows that
$\Col(\Pi)$ is solvable
\cite[Chapter~I, \S1, p.~775]{FeitThompson}.
\end{proof}

Equivalently, a projective plane of nonsquare order with even-order
collineation group has a flexible incidence graph, so the corresponding
right-angled Coxeter group is not graphically discrete. Hence a projective
plane with a nonflexible incidence graph has either square order or nonsquare
order with an odd solvable collineation group. Proposition~\ref{prop:plane-involutions}
therefore narrows the finite-geometric search but does not resolve the
existence question.

\subsection{Unbounded orbits from thick panels}

Let $X$ be a locally finite semi-regular right-angled building of type
$(W,S)$, where $S$ is finite. Here semi-regular means that every $r$-panel has
the same finite cardinality $q_r\geq2$, depending only on $r\in S$
\cite[\S1]{Caprace}. Write
$\Aut(X)^{+}$ for the type-preserving automorphism group. A building
automorphism is determined by its action on chambers and may therefore be
viewed as the induced chamber-graph automorphism. Because $S$ is finite, preservation of
each chamber-adjacency type is a closed condition in the pointwise-convergence
topology. Hence $\Aut(X)^{+}$ is a closed subgroup of the automorphism group of
the chamber graph. It is strongly transitive, and hence chamber-transitive, by
\cite[Proposition~6.1]{Caprace}.

For a panel $\sigma$, write $\Ch(\sigma)$ for its set of chambers. For a chamber
$x$, write $\proj_\sigma(x)$ for the gate projection of $x$ to $\sigma$.  Caprace's extension theorem
\cite[Proposition~4.2]{Caprace} states that every permutation $\alpha$ of
the chambers of a panel $\sigma$ extends to an element of $\Aut(X)^{+}$
that stabilizes $\sigma$, induces $\alpha$ on it, and fixes every chamber
whose projection to $\sigma$ is fixed by $\alpha$. Both propositions are stated
for semi-regular right-angled buildings and do not require $q_r>2$ for every
$r\in S$.

Under the rank-two hypothesis below, Kubena and Thomas proved that
both the type-preserving and full automorphism groups are
nondiscrete \cite[Theorem~1(1)]{KT}. The next proposition strengthens the
type-preserving conclusion to failure of compact-by-discreteness by exhibiting
unbounded chamber-stabilizer orbits.

\begin{proposition}\label{prop:thick-orbits}
Suppose that $q_s>2$ and $m_{st}=\infty$ for some $s,t\in S$.  Then
$\Aut(X)^{+}$ is not compact-by-discrete.
\end{proposition}

\begin{proof}
Fix a chamber $c$ and $k\geq1$. The incidence graph of the $s$- and
$t$-panels in the rank-two residue through $c$ is a tree \cite[\S1]{Caprace}. Choose a chamber $y$
so that a minimal gallery from $c$ to $y$ crosses, in order, distinct $s$-panels
\[
 \sigma_1,\ldots,\sigma_k.
\]
Put
\[
 a_i=\proj_{\sigma_i}(c),\qquad
 b_i=\proj_{\sigma_i}(y).
\]
The two projections $a_i$ and $b_i$ are distinct.  Since $q_s>2$, choose
$e_i\in\Ch(\sigma_i)\setminus\{a_i,b_i\}$.

Let $\alpha_i$ be the transposition of $b_i$ and $e_i$ fixing every other
chamber of $\sigma_i$.  Extend it by Caprace's theorem to
$g_i\in\Aut(X)^{+}$.  Since $\alpha_i(a_i)=a_i$, the automorphism $g_i$ fixes
$c$.

If $i<j$, the unique path in the rank-two incidence tree shows that every
chamber of $\sigma_i$ projects to $a_j$ on $\sigma_j$.  Hence $g_j$ fixes
$\sigma_i$ pointwise.  Equivariance of gate projections gives
\[
 \proj_{\sigma_i}(g_i y)=e_i,
 \qquad
 \proj_{\sigma_i}(g_j y)=b_i.
\]
Thus $g_i y\neq g_j y$ whenever $i<j$. Moreover,
$\proj_{\sigma_i}(g_i y)=e_i\neq b_i=\proj_{\sigma_i}(y)$, so
$g_i y\neq y$ for every $i$. Consequently
\[
 |\Aut(X)^{+}_c\cdot y|\geq k+1.
\]

The group $\Aut(X)^{+}$ is closed and chamber-transitive, hence acts
cocompactly on the connected locally finite chamber graph.  Since $k$ is
arbitrary, Theorem~\ref{thm:MSSW-orbits} proves the proposition.
\end{proof}

\subsection{The standard chamber graph of a graph product}

Let $\Gamma$ be a finite simplicial graph with vertex set $S$, and let
$(G_s)_{s\in S}$ be nontrivial finite groups. Their graph product
\cite[Chapter~3]{Green} is
\[
 G_\Gamma=\left(*_{s\in S}G_s\right)
 \big/\langle\!\langle[G_s,G_t]:\{s,t\}\in E\Gamma\rangle\!\rangle.
\]
By Shepherd's construction
\cite[Definitions~2.6 and~3.1, Lemmas~2.5(1) and~2.10]{Shepherd}, the standard
chamber graph $\cC_\Gamma$ has vertex set $G_\Gamma$. Two distinct vertices
$g,h$ are adjacent exactly when $g^{-1}h\in G_s\setminus\{1\}$ for some
$s\in S$. There is an associated right-angled building $\mathcal B_\Gamma$.
Its chamber graph is
$\cC_\Gamma$, its $s$-panels have cardinality $q_s:=|G_s|$, and $G_\Gamma$ acts
simply transitively on its chambers.

The discrete direction rests on the following product lemma. Its abstract
automorphism-group decomposition is a consequence of the Cartesian-product
results of Imrich \cite{ImrichCartesian} and
Miller \cite[Theorem, pp.~24--25]{MillerCartesian}.
The direct proof below also establishes the claimed topological-group isomorphism.

\begin{lemma}\label{lem:triangle-product}
Let $A$ and $B$ be connected graphs.  Suppose every edge of $A$ lies in a
triangle and $B$ is triangle-free.  With the pointwise-convergence
topologies, the map
\[
 (\alpha,\beta)\longmapsto\bigl((a,b)\longmapsto(\alpha(a),\beta(b))\bigr)
\]
is an isomorphism of topological groups from
$\Aut(A)\times\Aut(B)$ onto $\Aut(A\mathbin{\square}B)$.
In particular, the images of both factors are closed in the product
automorphism group. If
$A$ is finite and $\Aut(B)$ is discrete, then
$\Aut(A\mathbin{\square}B)$ is discrete.
\end{lemma}

\begin{proof}
If $A=K_1$ or $B=K_1$, the assertion is immediate, so assume both factors have
an edge. Every triangle in a Cartesian product lies in a layer of one factor.
Hence an edge of $A\mathbin{\square}B$ lies in a triangle if and only if it is
an $A$-edge. Every automorphism therefore preserves the $A$- and $B$-edges
separately. Since $A$ and $B$ are connected, the corresponding layers are the
connected components of the two single-edge-type subgraphs, so both families
of layers are permuted.

For $f\in\Aut(A\mathbin{\square}B)$, write
\[
 f(A\times\{b\})=A\times\{\beta(b)\},\qquad
 f(\{a\}\times B)=\{\alpha(a)\}\times B.
\]
Intersecting the two image layers gives
\[
 f(a,b)=(\alpha(a),\beta(b)).
\]
Preservation of the two edge types shows that $\alpha\in\Aut(A)$ and
$\beta\in\Aut(B)$. Conversely, every pair of factor automorphisms acts
on the product, so the displayed map is an abstract group isomorphism.

It remains to check the topology. The forward map is continuous by coordinate
evaluation. Fix $a_0\in V(A)$ and $b_0\in V(B)$. For an automorphism $f$ of
the product, the value $\alpha(a)$ is the first coordinate of $f(a,b_0)$,
and $\beta(b)$ is the second coordinate of $f(a_0,b)$. Thus the inverse is
continuous for the pointwise-convergence topologies. The map is therefore a
topological group isomorphism, and its two factors are closed. If $A$ is
finite, then $\Aut(A)$ is finite, giving the last assertion.
\end{proof}

In view of the Kubena--Thomas discreteness theorem
\cite[Theorem~1]{KT}, it remains to determine when compact-by-discreteness can
occur for the standard chamber graph.

\begin{theorem}\label{thm:chamber-classification}
Set
\[
 T=\{s\in S:|G_s|>2\},\qquad
 \Gamma_0=\Gamma[S\setminus T].
\]
The following conditions are equivalent.
\begin{enumerate}
\item $\Aut(\cC_\Gamma)$ is compact-by-discrete.
\item $\Aut(\cC_\Gamma)$ is discrete.
\item every vertex of $T$ is universal in $\Gamma$ and $\Gamma_0$ is
nonflexible.
\end{enumerate}
\end{theorem}

\begin{proof}
Suppose first that $s\in T$ is not universal, and choose $t\neq s$ not
adjacent to $s$.  Then $q_s>2$ and $m_{st}=\infty$ in $W_\Gamma$.
Proposition~\ref{prop:thick-orbits} shows that the closed subgroup
$\Aut(\mathcal B_\Gamma)^{+}$ of $\Aut(\cC_\Gamma)$ is not compact-by-discrete.
Lemma~\ref{lem:closed-subgroup} therefore implies that
$\Aut(\cC_\Gamma)$ is not compact-by-discrete.

Assume now that every vertex of $T$ is universal. The graph product splits as
a direct product, and the standard generating set is the disjoint union of the
generating sets of its two factors. Hence its standard chamber graph is the
Cartesian product
\[
 G_\Gamma=\left(\prod_{s\in T}G_s\right)\times W_{\Gamma_0},
 \qquad
 \cC_\Gamma=A\mathbin{\square}B,
\]
where
\[
 A=\mathop{\mathbin{\square}}_{s\in T}K_{q_s},
 \qquad
 B=\cC(W_{\Gamma_0},S\setminus T).
\]
If $T=\varnothing$, take $A=K_1$.  Otherwise, $A$ is finite and every edge
of $A$ lies in a triangle.  Coxeter word-length parity makes $B$ bipartite and hence triangle-free. Lemma~\ref{lem:triangle-product}
gives an isomorphism of topological groups
\[
 \Aut(\cC_\Gamma)\cong\Aut(A)\times\Aut(B).
\]

If $S\setminus T=\varnothing$, then $B=K_1$ and $\Aut(B)=\{1\}$. If
$S\setminus T\neq\varnothing$ and $\Gamma_0$ is nonflexible,
Theorem~\ref{thm:flexible} implies that $\Aut(B)$ is discrete. In either case $\Aut(\cC_\Gamma)$ is discrete.
If $\Gamma_0$ is flexible,
Theorem~\ref{thm:flexible} implies that $\Aut(B)$ is not
compact-by-discrete. Since $\Aut(B)$ is a closed direct factor,
Lemma~\ref{lem:closed-subgroup} shows that $\Aut(\cC_\Gamma)$ is not
compact-by-discrete. Finally, (2) implies~(1) because a discrete group is
compact-by-discrete. The three conditions are now equivalent.
\end{proof}

\begin{corollary}
If $G_\Gamma$ is graphically discrete, then every vertex of $T$ is universal
in $\Gamma$ and $\Gamma_0$ is nonflexible. If every vertex of $T$ is universal
in $\Gamma$, then
\[
 G_\Gamma=\left(\prod_{s\in T}G_s\right)\times W_{\Gamma_0}
\]
is graphically discrete if and only if $W_{\Gamma_0}$ is graphically
discrete.
\end{corollary}

\begin{proof}
The left action of $G_\Gamma$ on $\cC_\Gamma$ is geometric.  Thus failure of
condition~(3) in Theorem~\ref{thm:chamber-classification} supplies a witness
to failure of graphical discreteness.  Under the universal-vertex
hypothesis, the displayed splitting holds, and the last assertion follows
from Lemma~\ref{lem:finite-factor}.
\end{proof}

For right-angled Coxeter groups, $T=\varnothing$ and hence $\Gamma_0=\Gamma$.
In this case the flexible obstruction is already supplied by
Theorem~\ref{thm:flexible}.

\subsection{Odd graphs and Menger curve boundaries}

For $n\geq3$, the odd graph
\[
 O_n=\KG(2n-1,n-1)
\]
has vertex set consisting of the $(n-1)$-element subsets of
$\{1,\ldots,2n-1\}$. Two
vertices are adjacent when the corresponding subsets are disjoint \cite[p.~5]{EkinciGauci}. The graph
$O_3$ is the Petersen graph.

\begin{proposition}\label{prop:odd-graphs}
For every $n\geq3$, the graph $O_n$ is flexible, triangle-free, has no
induced four-cycle, is inseparable, and is nonplanar.
\end{proposition}

\begin{proof}
Fix a vertex $A$, so $|A|=n-1$.  A nontrivial permutation of the elements of
$A$, extended by the identity on the complement of $A$, induces an
automorphism of $O_n$. The induced automorphism is nontrivial because it moves
an $(n-1)$-subset containing one moved element but not its image. It fixes $A$
and fixes every neighbour of $A$ pointwise, since those neighbours are subsets
of the complement. Hence $O_n$ is flexible.

Three pairwise adjacent vertices would be three disjoint $(n-1)$-subsets of a
set of size $2n-1$, which is impossible for $n\geq3$.  Thus $O_n$ is
triangle-free.  If two distinct vertices $A$ and $C$ had two distinct common
neighbours, the complement of $A\cup C$ would contain two distinct
$(n-1)$-subsets and would therefore have at least $n$ elements.  But
$|A\cup C|\geq n$, so its complement has at most $n-1$ elements.  This
contradiction shows that $O_n$ has no four-cycle.

For the Kneser graph $\KG(N,r)$ with $N>2r$, every vertex has degree
$\binom{N-r}{r}$. Its vertex connectivity is also
$\binom{N-r}{r}$ \cite[Theorem~2]{EkinciGauci}. Thus $O_n$ is $n$-connected. Dani--Haulmark--Walsh note that a triangle-free graph
is inseparable exactly when it is connected and has no separating vertex,
separating edge, cut pair, or separating vertex suspension
\cite[p.~136]{DHW}. For $n\geq4$, deletion of at most three vertices does
not disconnect $O_n$, so all four obstructions are absent. For $n=3$, the
Petersen graph is $3$-connected, and a direct check shows that deleting the
three vertices of any induced path of length two leaves a connected graph.
The Dani--Haulmark--Walsh criterion gives inseparability.

The Petersen graph is nonplanar.  For $n\geq4$, the graph $O_n$ is
$n$-regular and has girth at least five.  If it were planar, Euler's formula
would give
\[
 |E(O_n)|\leq\frac{5}{3}(|V(O_n)|-2),
\]
whereas regularity gives $|E(O_n)|=n|V(O_n)|/2\geq2|V(O_n)|$, a
contradiction.  Hence every $O_n$ is nonplanar.
\end{proof}

\begin{corollary}\label{cor:menger-odd}
For every $n\geq3$, the right-angled Coxeter group $W_{O_n}$ is hyperbolic,
its Gromov boundary is homeomorphic to the Menger curve, and it is not
graphically discrete.
\end{corollary}

\begin{proof}
A right-angled Coxeter group is hyperbolic if and only if its defining graph
has no induced four-cycle \cite[Corollary~12.6.3]{Davis}.  Thus $W_{O_n}$ is hyperbolic by
Proposition~\ref{prop:odd-graphs}. Since $W_{O_n}$ is hyperbolic and $O_n$ is
triangle-free, inseparable and nonplanar, Dani--Haulmark--Walsh's
Menger-curve criterion
\cite[Corollary~1.7]{DHW} shows that $\partial W_{O_n}$ is homeomorphic to the
Menger curve. Finally, the flexibility of $O_n$ and
Theorem~\ref{thm:flexible} show that $W_{O_n}$ is not graphically discrete.
\end{proof}

For these examples with Menger curve boundary, failure of graphical discreteness is
already witnessed by the standard Coxeter Cayley graph, without passing to a
thick building. They also complement the generic picture suggested by
\cite[\S1.5, Conjecture~1.9]{MSSW}, which predicts that random groups in the
few-relator and Gromov density models are graphically discrete asymptotically
almost surely as the relator length tends to infinity.

\section{Questions}\label{sec:questions}

The present arguments do not decide whether the graphically discrete case of
Theorem~\ref{thm:polygon-classification} is nonvacuous.

\begin{question}\label{q:polygon-existence}
Does there exist a finite thick generalized $m$-gon with $m\geq3$ whose
incidence graph is nonflexible?
\end{question}

For projective planes, Proposition~\ref{prop:plane-perspectivity} reduces
the $m=3$ case to finding a perspectivity-free plane, while
Proposition~\ref{prop:plane-involutions} restricts any such example to square
order or to nonsquare order with an odd solvable collineation group. These
restrictions do not settle Question~\ref{q:polygon-existence}.

More generally, Corollary~\ref{cor:separating-star}
rules out graphical discreteness for every $W_\Gamma$ admitting a proper partial
conjugation, while Theorem~\ref{thm:doubling-obstruction} excludes additional
groups whenever a one-vertex double is flexible.
These obstructions suggest the following group-level problem.

\begin{question}
Which right-angled Coxeter groups $W_\Gamma$ are graphically discrete when
$\Gamma$ and every one-vertex double $D_v\Gamma$ are nonflexible?
\end{question}

\end{document}